\documentclass[smallextended,final,numbook]{svjour3}     % onecolumn (ditto)
\usepackage[colorlinks=true, linkcolor=blue, citecolor=blue, urlcolor=blue, pdfborder={0 0 0}]{hyperref}
\usepackage[letterpaper,top=2in, bottom=1.5in, left=1in, right=1in]{geometry}

\usepackage{amssymb}
\usepackage{tikz}
\graphicspath{ {./figures/} }
\usepackage[%
    font={small,sf},
    labelfont=bf,
    format=hang,
    format=plain,
    margin=0pt,
    width=0.8\textwidth,
]{caption}
\usepackage[list=true]{subcaption}
\usepackage{xcolor,colortbl}
\usepackage{mathtools}
\usepackage[normalem]{ulem} % to use \sout
\newtheorem{algo}{Algorithm}

\newcommand{\remove}[1]{{}}

\newcommand{\vN}{{\mathbf{N}}}

\newcommand{\vR}{{\mathbf{R}}}

\newcommand{\cB}{{\mathcal{B}}}

\newcommand{\cG}{{\mathcal{G}}}

\newcommand{\cL}{{\mathcal{L}}}

\newcommand{\cP}{{\mathcal{P}}}

\newcommand{\cT}{{\mathcal{T}}}

\newcommand{\sign}{\mathrm{sign}}

\newcommand{\prox}{\mathbf{prox}}

\DeclareMathOperator*{\argmin}{arg\,min}

\DeclareMathOperator*{\Min}{minimize}

\newcommand{\bc}{\begin{center}}
\newcommand{\ec}{\end{center}}

\newcommand{\bdm}{\begin{displaymath}}
\newcommand{\edm}{\end{displaymath}}

\newcommand{\beq}{\begin{equation}}
\newcommand{\eeq}{\end{equation}}

\newcommand{\bfl}{\begin{flushleft}}
\newcommand{\efl}{\end{flushleft}}

\newcommand{\bt}{\begin{tabbing}}
\newcommand{\et}{\end{tabbing}}

\newcommand{\beqn}{\begin{align}}
\newcommand{\eeqn}{\end{align}}

\newcommand{\beqs}{\begin{align*}} % no equation numbers
\newcommand{\eeqs}{\end{align*}}  % no equation numbers

\usepackage[lined,boxed,commentsnumbered, ruled,vlined]{algorithm2e}

\newcommand\numberthis{\addtocounter{equation}{1}\tag{\theequation}}

\DeclarePairedDelimiter{\dotp}{\langle}{\rangle}

\usepackage{booktabs}

\newif\iftechreport % declare the switch
\techreporttrue     % set the switch to true

\def\cut#1{{}}
\smartqed

\title{An $O(n\log(n))$ Algorithm for Projecting Onto the Ordered Weighted $\ell_1$ Norm Ball\thanks{This work is supported in part by NSF grant DMS-1317602.
}}

\author{Damek Davis}

\institute{D. Davis \at
              Department of Mathematics, University of California, Los Angeles\\
              Los Angeles, CA 90025, USA\\
              \email{damek@ucla.edu}}

\date{\today}
\journalname{Report} % Springer stuff

\begin{document}

\maketitle
\abstract{
The ordered weighted $\ell_1$ (OWL) norm is a newly developed generalization of the Octogonal Shrinkage and Clustering Algorithm for Regression (OSCAR) norm. This norm has  desirable statistical properties and can be used to perform simultaneous clustering and regression. In this paper, we show how to compute the projection of an $n$-dimensional vector onto the OWL norm ball in $O(n\log(n))$ operations. In addition, we illustrate the performance of our algorithm on a synthetic regression test. 

}

% SIAM \slugger{siopt}{xxxx}{xx}{x}{x--x}%slugger should be set to mms, siap, sicomp, sicon, sidma, sima, simax, sinum, siopt, sisc, or sirev

% SIAM \begin{abstract} \end{abstract}

% SIAM \begin{keywords} \end{keywords}

% SIAM \begin{AMS} 47H05, 65K05, 65K15, 90C25 \end{AMS}

% SIAM \pagestyle{myheadings}
% SIAM \thispagestyle{plain}
% SIAM \markboth{D. Davis, W. Yin}{}
\section{Introduction}

Sparsity is commonly used as a model selection tool in statistical and machine learning problems. For example, consider the following Ivanov regularized (or constrained) regression problem:
\begin{align}\label{eq:l0}
\Min_{x \in \vR^n} &\; \frac{1}{2}\|Ax - b\|^2 \quad \quad\text{subject to:} \; \|x\|_0\leq \varepsilon.
\end{align}
where $m, n >0$ are integers, $\varepsilon > 0$ is a real number, $A \in \vR^{m \times n}$ and $b \in \vR^m$ are given, and $\|x\|_0$ is the number of nonzero components of a vector $x\in \vR^n$. Solving~\eqref{eq:l0} yields the ``best" predictor $x$ with fewer than $\varepsilon$ nonzero components. Unfortunately,~\eqref{eq:l0} is nonconvex and NP hard~\cite{l0hard}. Thus, in practice the following convex surrogate (LASSO) problem is solved instead (see e.g.,~\cite{compsens}):
\begin{align}\label{eq:l1}
\Min_{x \in \vR^n} &\; \frac{1}{2}\|Ax - b\|^2 \quad \quad\text{subject to:} \; \|x\|_1\leq \varepsilon
\end{align}
where $\|x\|_1 = \sum_{i=1}^n |x_i|$. 

Recently, researchers have moved beyond the search for sparse predictors and have begun to analyze ``group-structured" sparse predictors~\cite{bach2012}.  These models are motivated by a deficiency of~\eqref{eq:l0} and~\eqref{eq:l1}: they yield a predictor with a small number of nonzero components, but they fail to identify and take into account similarities between features. In other words, group-structured predictors simultaneously cluster and select groups of features for prediction purposes. Mathematically, this behavior can be enforced by replacing the $\ell_0$ and $\ell_1$ norms in~\eqref{eq:l0} and~\eqref{eq:l1} with new regularizers. Typical choices for group-structured regularizers include the Elastic Net~\cite{RSSB:RSSB503} (EN), Fused LASSO~\cite{RSSB:RSSB490},  Sparse Group LASSO~\cite{doi:10.1080/10618600.2012.681250}, and Octogonal Shrinkage and Clustering Algorithm for Regression~\cite{BIOM:BIOM843} (OSCAR). The EN and OSCAR regularizers have the benefit of being invariant under permutations of the components of the predictor and do not require prior specification of the desired groups of features (when a clustering is not known \textit{a priori}). However, OSCAR has been shown to outperform EN regularization in feature grouping~\cite{BIOM:BIOM843,6238378}. This has motivated the recent development of the ordered weighted $\ell_1$ norm~\cite{bogdan2013statistical,zeng2014decreasing} (OWL) (see~\eqref{eq:OWL} below), which includes the OSCAR, $\ell_1$, and $\ell_\infty$ norms as a special case. 

\textbf{Related work.} Recently, the paper~\cite{zeng2014ordered} investigated the properties of the OWL norm, discovered the atomic norm characterization of the OWL norm, and developed an $O(n\log(n))$ algorithm for computing its proximal operator (also see~\cite{bogdan2013statistical} for the computation of the proximal operator). Using the atomic characterization of the OWL norm, the paper~\cite{zeng2014ordered} showed how to apply the Frank-Wolfe conditional gradient algorithm (CG)~\cite{FWalg} to the Ivanov regularized OWL norm regression problem. However, when more complicated, and perhaps, nonsmooth data fitting and regularization terms are included in the Ivanov regularization model, the Frank-Wolfe algorithm can no longer be applied. If we knew how to quickly project onto the OWL norm ball, we could apply modern proximal-splitting algorithms~\cite{combettes2011proximal}, which can perform better than CG for OWL problems~\cite{zeng2014ordered}, to get a solution of modest accuracy quickly. Note that~\cite{zeng2014ordered} proposes a root-finding scheme for projecting onto the OWL norm ball, but it is not guaranteed to terminate at an exact solution in a finite number of steps.

\textbf{Contributions.} The paper introduces an $O(n \log(n))$ algorithm and MATLAB code for projecting onto the OWL norm ball (Algorithm~\ref{alg:basic}). Given a norm $f: \vR^n \rightarrow \vR_+$, computing the proximal map 
\begin{align*}
\prox_{ f}(z) &:= \argmin_{x \in \vR^n} f(x) + \frac{1}{2} \|x - z\|^2
\end{align*}
can be significantly easier than evaluating the projection map
\begin{align}\label{eq:normprojection}
P_{\{x \in \vR^n | f(x) \leq \varepsilon\}}(z) := \argmin_{f(x) \leq \varepsilon} \frac{1}{2}\|x - z\|^2.
\end{align}
In this paper, we devise an $O(n\log(n))$ algorithm to project onto the OWL norm ball that matches the complexity (up to constants) of the currently best performing algorithm for computing the proximal operator of the OWL norm. The algorithm we present is the first known method that computes the projection in a finite number of steps, unlike the existing root-finding scheme~\cite{zeng2014ordered}, which only provides an approximate solution in finite time. In addition, using duality (see~\eqref{eq:dualowlprox}) we immediately get an $O(n\log(n))$ algorithm for computing the proximity operator of the dual OWL norm (see~\eqref{eq:dualOwlNorm}).

The main bottleneck in evaluating the proximity and projection operators of the OWL norm arises from repeated partial sortings and averagings. Unfortunately, this seems unavoidable because even evaluating the OWL norm requires sorting a (possibly) high-dimensional vector. This suggests that any OWL norm projection algorithm requires $\Omega(n\log(n))$ operations in the worst case.% After the initial sort, both algorithms iteratively group the components of the sorted vector and effectively discover clusters.  The main difference between the algorithms that compute the proximal and projection operators is the way that clusters are merged throughout the algorithm.

\textbf{Organization.} The OWL norm is introduced in Section~\ref{sec:defs}. In Section~\ref{sec:trans}, we reduce the OWL norm projection to a simpler problem (Problem~\ref{prob:reduced}). In Section~\ref{sec:parts}, we introduce crucial notation and properties for working with partitions. In Section~\ref{sec:alg}, we introduce the 6 alternatives (Proposition~\ref{prop:incrgamma}), which directly lead to our main algorithm (Algorithm~\ref{alg:basic}) and its complexity (Theorem~\ref{thm:convergence}). Finally, in Section~\ref{sec:numerical}, we illustrate the performance of our algorithm on a synthetic regression test. 

\section{Basic Properties and Definitions}\label{sec:defs}
We begin with the definition of the OWL norm.
\begin{definition}[The OWL Norm]
Let $n \geq 1$ and let $w \in \vR^{n}_{+}$ satisfy $w_1 \geq w_{2} \geq \cdots w_{n} \geq 0$ with $w \neq 0$.  Then for all $z \in \vR^n$, the OWL norm $\Omega_w : \vR^n \rightarrow \vR_{+}$ is given by
\begin{align}\label{eq:OWL}
\Omega_w(x) := \sum_{i=1}^n w_i |x|_{[i]}
\end{align}
where for any $x \in \vR^n$, the scalar $|x|_{[i]}$ is the $i$-th largest element of the magnitude vector $|x| := (|x_1|, \ldots, |x_n|)^T$. For all $\varepsilon > 0$, let $\cB(w, \varepsilon) := \{ x\in \vR^n \mid \Omega_w(x) \leq \varepsilon\}$ be the closed OWL norm ball of radius $\varepsilon$.
\end{definition}

Notice that when $w$ is a constant vector, we have $\Omega_w \equiv w_1 \|\cdot\|_1$.  On the other hand, when $w_1 = 1$ and $w_{i} = 0$ for $i = 2, \ldots, n$, we have $\Omega_w \equiv w_1\|\cdot \|_\infty$. Finally, given nonnegative real numbers $\mu_1$ and $\mu_2$, for all $i \in \{1, \ldots, n\}$, define $w_i = \mu_1 + \mu_2(n-i)$. Then the OSCAR norm~\cite{bogdan2013statistical} is precisely: 
\begin{align}\label{eq:OSCAR}
\Omega_w(x) = \mu_1\|x\|_1 + \mu_2\sum_{i< j} \max\{|x_i|, |x_j|\}.
\end{align}

Note that $\Omega_w$ was originally shown to be a norm in~\cite{bogdan2013statistical,zeng2014decreasing}. The paper~\cite{zeng2014decreasing} also showed that the dual norm (in the sense of functional analysis) of $\Omega_w$ has the form
\begin{align}\label{eq:dualOwlNorm}
\Omega_w^\ast(x) = \max\{ \tau_i \|x_{(i)}\|_1 \mid i = 1, \ldots, n\}
\end{align}
where $x\in \vR^n$ and for all $1 \leq j \leq n$, $\tau_j = \left(\sum_{i=1}^j w_j\right)^{-1}$ and $x_{(j)} \in \vR^j$ is a vector consisting of the $j$ largest components of $x$ (where size is measured in terms of magnitude). One interesting consequence of this fact is that for all $\gamma > 0$ and $z \in \vR^n$, we have (from \cite[Proposition 23.29]{bauschke2011convex})
\begin{align}\label{eq:dualowlprox}
\prox_{\gamma \Omega_w^\ast}(z) := \argmin_{x\in \vR^n} \left\{ \Omega_w^\ast(x) + \frac{1}{2\gamma}\|x-z\|^2\right\} = z - \gamma P_{\cB(w, 1)}\left(\frac{1}{\gamma} z\right).
\end{align}
Thus, Algorithm~\ref{alg:basic} (below) also yields an $O(n\log(n))$ algorithm for evaluating $\prox_{\gamma \Omega_w^\ast}(z)$.

\subsection{A Simplification of the OWL Norm Projection Problem}\label{sec:trans}
The following transformation (which is based on~{\cite[Lemmas 2-4]{zeng2014ordered}}) will be used as a preprocessing step in our algorithm. For convenience, we let $\odot : \vR^n \times \vR^n \rightarrow \vR^n$ denote the componentwise vector product operator. Finally, for any $z\in \vR^n$, let $\sign(z)\in \{-1, 1\}^n$ be the componentwise vector of signs of $z$ (with the convention $\sign(0) = 1$).

\begin{proposition}[Problem Reduction]\label{prop:sort}
Let $z \in \vR^n$, and let $Q(|z|)$ be the permutation matrix that sorts  $|z|$ to be in nonincreasing order. Then 
\begin{align*}
P_{\cB(w,\varepsilon)}(z) = \sign(z)\odot Q(|z|)^T P_{\cL(w, \varepsilon) \cap \cT} (Q(|z|) |z|)
\end{align*}
where $\cL(w, \varepsilon) := \{ x\in \vR^n \mid \dotp{w, x} \leq \varepsilon\}$ and $\cT := \{x \in \vR^n \mid x_1 \geq x_2 \geq \cdots x_n \geq 0 \}$.
\end{proposition}
\begin{proof}
Note that $\Omega_w(\sign(z)\odot Q(|z|)x) = \Omega_w(x)$ for all $x \in \vR^n$. Thus, 
\begin{align*}
&P_{\cB(w, \varepsilon)} (Q(|z|) |z|) = P_{\cB(w,\varepsilon)}(\sign(z) \odot Q(|z|)z) = \argmin_{\Omega_w(x) \leq \varepsilon} \frac{1}{2}\left\|\sign(z) \odot Q(|z|)z - x\right\|^2 \\
&\hspace{20pt}= \argmin_{\Omega_w(x) \leq \varepsilon} \frac{1}{2}\left\|z - \sign(z) \odot Q(|z|)^Tx\right\|^2 = \sign(|z|)\odot Q(|z|)P_{\cB(w, \varepsilon)}( z).
\end{align*}
Thus, we have shown that for general vectors $z \in \vR^n$, we have $P_{\cB(w,\varepsilon)}(z) = \sign(z) \odot Q(|z|)^TP_{\cB(w,\varepsilon)}(Q(|z|)|z|) $. Finally, the result follows from the equality $P_{\cB(w,\varepsilon)}(Q(|z|)|z|) = P_{\cL(w, \varepsilon)\cap \cT}(Q(|z|)|z|)$.
%First note that $\sign(x^\ast) = \sign(z)$. Indeed, this follows from~\eqref{eq:normprojection}, the equality $\Omega_w(\sign(z)x^\ast) = \Omega_w(x^\ast)$, and the inequality $\|z - \sign(z)\odot|x^\ast|\|^2 \leq \|z - x^\ast\|^2$. Thus, $P_{\cB(w,\varepsilon)}(z) = \sign(z)\odot P_{\cB(w,\varepsilon)}(|z|)$. Thus, without loss of generality we now assume that $z = |z|$, and hence, $x^\ast = |x^\ast|$. 
%
%Next, we claim that $Q(z) x^\ast  = P_{\cB(w,\varepsilon)}(Q(z)z).$ Indeed, this follows from~\eqref{eq:normprojection}, the equality $\Omega_w(Q(z)x^\ast) = \Omega_w(x^\ast)$, and the equality $\|Q(z)z - Q(z)x^\ast\|^2 = \|z - x^\ast\|^2$ (because $Q(z)$ is an orthogonal matrix). 
%
%Thus, we have shown that for general vectors $z \in \vR^n$, we have $P_{\cB(w,\varepsilon)}(z) = \sign(z) \odot Q(|z|)^TP_{\cB(w,\varepsilon)}(Q(|z|)|z|) $. Finally, the result follows from the equality $P_{\cB(w,\varepsilon)}(Q(|z|)|z|) = P_{\cL(w, \varepsilon)\cap \cT}(Q(|z|)|z|)$.
\end{proof}

Thus, whenever $z \in \cT$, projecting onto the OWL norm ball is equivalent to projecting onto the set intersection $\cL(w, \varepsilon)\cap \cT$:
\begin{align*}
P_{\cB(w,\varepsilon)}(z) =&  \argmin_{x\in \vR^n} \frac{1}{2}\|x - z\|^2 \quad \text{subject to: } \; \sum_{i = 1}^n w_i x_i \leq \varepsilon \text{ and } x_1 \geq x_2 \geq \cdots x_n \geq 0.
\end{align*}

Finally, we make one more reduction to the problem, which is based on the following simple lemma.
\begin{lemma}
Let $z, w \in \cT$ and suppose that $w \neq 0$. If $\dotp{z, w} \leq \varepsilon$, then $P_{\cB(w,\varepsilon)}(z) = z$.  Otherwise, $\dotp{P_{\cB(w,\varepsilon)}(z), w} = \varepsilon$.
\end{lemma}

We arrive at our final problem: 
\begin{problem}[Reduced Problem]\label{prob:reduced}
Given $z\in \cT$ such that $\dotp{z, w} > \varepsilon$, find 
\begin{align}\label{eq:mainproblem}
x^\ast &:= \argmin_{x\in \vR^n} \frac{1}{2}\|x - z\|^2 \quad \text{subject to: } \; \sum_{i = 1}^n w_i x_i = \varepsilon \text{ and } x_1 \geq x_2 \geq \cdots x_n \geq 0
\end{align}
\end{problem}
Now define $H(w, \varepsilon) = \{x \in \vR^n \mid \dotp{w, x} = \varepsilon\}$. Then $x^\ast = P_{H(w, \varepsilon)\cap\cT}(z)$.  

The following proposition is a straightforward exercise in convex analysis.

\begin{proposition}[KKT Conditions]\label{prop:optimality}
The point $x^\ast$ satisfies Equation~\eqref{eq:mainproblem} if, and only if, there exists $\lambda^\ast \in \vR_{++}$ and a vector $v^\ast \in \vR^n_{+}$ such that
\begin{enumerate}
\item $x^\ast \in \cT$
\item $v^\ast_i(x_i^\ast - x_{i+1}^\ast) = 0$ for $1 \leq i < n$ and $v^\ast_n x_n = 0$;
\item $x_i^\ast = z_i - \lambda^\ast w_i + v^\ast_i- v^\ast_{i-1}$ for $1 \leq i \leq n$ where $v_0^\ast := 0$;
\item and $\dotp{x^\ast, w} = \varepsilon.$
\end{enumerate}
\end{proposition}
%For convenience we let $v^\ast_{0}  = 0$, throughout the result of the paper. 

We now record the solution to~\eqref{eq:mainproblem} in the special case that $w$ is the constant vector.

\begin{proposition}[Projection Onto the Simplex {\cite{original_simplex}}]\label{prop:simplex}
Let $\kappa > 0$ and let $\Delta(\kappa, n)$ denote the simplex $\{x \in \vR^n \mid 0 \leq x \leq \kappa \text{ and }  \sum_{i=1}^n x_i = \kappa \}$. Let $z, w \in \cT$ and suppose that $w \neq 0$. In addition, suppose that $w_1 = w_2 = \cdots = w_n$. Then $x^\ast = P_{\Delta(\varepsilon/w_1, n)}(z)$ is the solution to Problem~\eqref{eq:mainproblem}. In other words, we can replace the constraint $x \in \cT$ with $x\in \vR_+^n$ in Problem~\eqref{eq:mainproblem}.  Furthermore, $x^\ast = \max\{z - \lambda , 0\}$ where \begin{align*}
\lambda := \frac{ \sum_{i=1}^K z_i - \varepsilon/w_1 }{K}  &&\text{and} && K:=\max\left\{k \mid \frac{ \sum_{i=1}^k z_i - \varepsilon/w_1 }{k} < z_k\right\}.
\end{align*}
\end{proposition}

\subsection{Partitions}\label{sec:parts}
Define $\cP_n$ to be the set of partitions of $\{1, \ldots, n\}$ built entirely from intervals of integers. For example, when $n = 5$, the partition $\cG := \{\{1, 2\}, \{3\}, \{4, 5\}\}$ is an element of $\cP_5$, but $\cG' := \{\{1, 3\}, \{2, 4, 5\}\}$ is not an element of $\cP_5$ because $\{1, 3\}$ and $\{2, 4, 5\}$ are not intervals. For two partitions $\cG_1, \cG_2 \in \cP_n$, we say that $$\text{$\cG_{1} \preccurlyeq \cG_{2}$ if for all $G_1 \in \cG_{1}$, there exists $G_2 \in \cG_{2}$ with $G_1 \subseteq G_2$}.$$
Note that if $\cG_1 \preccurlyeq \cG_2$ and $\cG_2 \preccurlyeq \cG_1$, then $\cG_1 = \cG_2$. In addition, we have the following fact:
\begin{lemma}\label{eq:equalcardinality}
Let $\cG_1, \cG_2 \in \cP_n$. If $\cG_1 \preccurlyeq \cG_2$ and $|\cG_1| = |\cG_2|$, then $\cG_1 = \cG_2$. 
\end{lemma}

Suppose that we partition a vector $z \in \vR^n$ into $g$ maximal groups of nondecreasing components 
\begin{align*}
z = (\underbrace{z_1, \ldots, z_{n_1}}_{G_{1}(z)}, \underbrace{z_{n_1 + 1}, \ldots, z_{n_2}}_{G_{2}(z)}, \ldots, \underbrace{z_{n_{g-1} + 1}, \ldots, z_{n_{g}}}_{G_{g}(z)})^T
\end{align*}
where $z_{n_j} > z_{n_{j} + 1}$ for all $1 \leq j \leq g-1$, and inside the each group, $z$ is a nondecreasing list of numbers (i.e., $z_k \leq z_{k+1}$ whenever $k, k+1\in G_j(z)$ for some $j \in \{1, \cdots, g\}$). Note that $g$ can be $1$, in which case we let $n_0 = 1$.  We let 
\begin{align}\label{eq:vector-to-group}
\cG(z) := \{G_1(z), \ldots, G_g(z)\} \in \cP_n. 
\end{align} 
For example, for $z :=  (1, 4, 5,1,3)^T$, we have $\cG(z) = \{\{1, 2, 3\}, \{4, 5\}\}$, $g = 2$, $G_1(z) = \{1, 2, 3\}$, and $G_2(z) = \{4, 5\}$. Note that when $z \in \cT$, the vector $z$ is constant within each group. 

For simplicity, whenever $x^\ast$ is a solution to~\eqref{eq:mainproblem}, we define 
\begin{align}\label{eq:optimalgroup}
\cG^\ast := \cG(x^\ast).
\end{align}
Finally, for simplicity, we will also drop the dependence of the groups on $z$: $G_i := G_i(z)$.

For any vector $z \in \vR^n$ and any partition $\cG = \{G_1, \ldots, G_g\} \in \cP_n$, define an averaged vector: for all $j = 1, \ldots, g$ and $i \in G_j$, let
\begin{align}\label{eq:averagedidentity}
(z_\cG)_i := \frac{1}{|G_j|}\sum_{k \in G_j} z_k.
\end{align}
For example, for $z :=  (1, 4, 5,1,3)^T$ and $\cG := \{\{1, 2\}, \{3, 4\}, \{5\}\}$, we have $z_\cG = (5/2, 5/2, 3, 3, 5)^T$. Note that $z_\cG \in \cT$ whenever $z \in \cT$. 

The following proposition will allow us to repeatedly apply transformations to the vectors $z$ and $w$ without changing the solution to~\eqref{eq:mainproblem}. %Every time we apply these transformations, the partition $\cG(z)$ decreases in size, which will help us show finite convergence (in at most $n$ steps) of our algorithm.
\begin{proposition}[Increasing Partitions]\label{prop:equivalence}
Let $z, w \in \cT$ and suppose that $w \neq 0$.
\begin{enumerate}
\item \label{prop:equivalence:part:1} Suppose that $\lambda^\ast \geq \lambda$ (where $\lambda^\ast$ is as in Proposition~\ref{prop:optimality}). Then we have $$\cG(z) \preccurlyeq \cG(z - \lambda w) = \cG(z_{\cG(z - \lambda w)} )\preccurlyeq \cG^\ast.$$
\item \label{prop:equivalence:part:2} We have $x^\ast = P_{H(w_\cG, \varepsilon)\cap \cT}(z_\cG)$ whenever $\cG \preccurlyeq \cG^\ast$. %}We have $x^\ast = P_{H_{w, \varepsilon} \cap \cT}(z) = P_{H_{w_{\cG(z- \lambda w)}, \varepsilon} \cap \cT}(z_{\cG(z- \lambda w)}) = P_{H_{w_{\cG(z)}, \varepsilon} \cap \cT}(z) $.
\end{enumerate}
\end{proposition}
\begin{proof}
\iftechreport
  See Appendix~\ref{app:prop:equivalence}.
\else
  See the supplementary material. 
\fi
\iftechreport
  \qed
\else 
\fi 
\end{proof}

\section{The Algorithm}\label{sec:alg}

The following proposition is the workhorse of our algorithm. It provides a set of 6 alternatives, three of which give a solution when true; the other three allow us to update the vectors $z$ and $w$ so that $\cG(z)$ strictly decreases in size, while keeping $x^\ast$ fixed. Clearly, the size of this partition must always be greater than 0, which ensures that our algorithm terminates in a finite number of steps.
\begin{proposition}[The 6 Alternatives]\label{prop:incrgamma}
Let $z, w \in \cT$. Suppose that $w \neq 0$, that $\dotp{w, z} > \varepsilon$, and  $w = w_{\cG(z)}$. Let $$r := \min\left\{\frac{z_i - z_{i+1}}{w_i - w_{i+1}} \mid i = 1, \ldots, n-1\right\}$$ where we use the convention that $0/0 = \infty$. Define 
\begin{align*}
\lambda_0 := \frac{\left(\sum_{\{i \mid z_i > z_n\}} z_i w_i\right) - \varepsilon}{\sum_{\{i \mid z_i > z_n\}} w_i^2} && \text{and} && \lambda_1 := \frac{\dotp{z, w} - \varepsilon}{\|w\|^2}.
\end{align*}
Then $\lambda^\ast \geq \lambda_1$ (where $\lambda^\ast$ is as in Proposition~\ref{prop:optimality}).

Let $n' := \min\left(\{ k \mid z_k - \lambda_0 w_k < 0\}\cup \{n+1\}\right)$.  Then one of the following mutually exclusive alternatives must hold: 
\begin{enumerate}
\item \label{prop:incrgamma:simplex} If $r = \infty$, we have $x^\ast = P_{\Delta(\varepsilon/w_1, n)}(z)$.
\item \label{prop:incrgamma:lambda_1} If $\lambda_1 > r$, then $\lambda^\ast \geq \lambda_1 > r$.
\item \label{prop:incrgamma:lambda_1z_n}If $\lambda_1 \leq r< \infty$ and $z_n - \lambda_1 w_{n} \geq 0$, then $x^\ast = z - \lambda_1 w$.
\item \label{prop:incrgamma:lambda_1z_nlambda_0g} If $\lambda_1 \leq r < \infty $, $z_n - \lambda_1 w_{n} < 0$ and $\lambda_0 > r$, then $\lambda^\ast \geq \lambda_0 > r$.
\item \label{prop:incrgamma:lambda_1z_nlambda_0l} If $\lambda_1 \leq r< \infty$, $z_n - \lambda_1 w_{n} < 0$, $\lambda_0 \leq r$, and $n' \leq n$ with $z_{n'} = z_n$, then $x^\ast = \max\{z - \lambda_0w, 0\}$.
\item  \label{prop:incrgamma:G_0} If $\lambda_1 \leq r < \infty$, $z_n - \lambda_1 w_{n} < 0$, $\lambda_0 \leq r$, and   $n' < n$ with  $z_{n'} \neq z_n$, then $\cG_0 \preccurlyeq \cG^\ast$ where $\cG_0 = \{G \in \cG(z) \mid \max(G) < n'\} \cup \{ \{n', \ldots, n\}\}$.
\item \label{prop:incrgamma:allfail} It cannot be the case that $\lambda_1 \leq r < \infty $, $z_n - \lambda_1 w_{n} < 0$, $\lambda_0 \leq r$, and  $n' = n+1$.
\end{enumerate}
In addition, whenever $\lambda^\ast \geq \lambda \geq r$, we have $\cG(z - \lambda w) \preccurlyeq \cG^\ast$ and $|\cG(z_{\cG(z - \lambda w)})| \leq |\cG(z)| - 1$. Similarly, when~\ref{prop:incrgamma:G_0} holds, we have $\cG_0 \in \cP_n$, $\cG(z) \preccurlyeq \cG_0 = \cG(z_{\cG_0}) \preccurlyeq \cG^\ast$, and $|\cG(z_{\cG_0})| \leq |\cG(z)| - 1$. In particular, if $\cG(z) = \cG^\ast$, then at least one of steps~\ref{prop:incrgamma:simplex},~\ref{prop:incrgamma:lambda_1z_n}, and \ref{prop:incrgamma:lambda_1z_nlambda_0l} will not fail.
\end{proposition}
%\begin{enumerate}
%\item \label{prop:incrgamma:test0} If $r = \infty$, we have $x^\ast = P_{\Delta(\varepsilon/w_1, n)}(z)$.
%\item \label{prop:incrgamma:test1} If $\infty > r \geq \lambda_0$ and 
%\begin{align*}
%z_n - \lambda_0 w_n &\leq 0 \\
%z_{n'} - \lambda_0 w_{n'} &\geq 0 \quad\quad \text{where } n' = \max\{k \mid z_k > z_n\}
%\end{align*}
%then $x^\ast  = \max\{z - \lambda_0 w, 0\}$
%\item \label{prop:incrgamma:test2} If $\infty > r \geq \lambda_1$ and $z_n - \lambda_1 w_n \geq 0$, then $x^\ast = z - \lambda_1 w$.
%\item \label{prop:incrgamma:test3} If none of the above hold, then $\lambda^\ast_{(z, w)} > r$ and $|\cG(z_{\cG(z - rw)})| \leq |\cG(z)| - 1$.
%\end{enumerate}
\begin{proof}
\iftechreport
See Appendix~\ref{app:prop:incrgamma}.
\else
  See the supplementary material. 
\fi
\iftechreport
  \qed
\else 
\fi 
\end{proof}

We are now ready to present our algorithm. It repeatedly transforms the vectors $z$ and $w$ after checking whether Proposition~\ref{prop:incrgamma} yields a solution to~\eqref{eq:mainproblem}. Note that we assume the input is sorted and nonnegative. Thus to project onto the OWL ball with Algorithm~\ref{alg:basic}, the preprocessing in Proposition~\ref{prop:sort} must be applied first. Please see 
\iftechreport
Appendix~\ref{app:example}
\else
 the supplementary material
\fi
for an example of Algorithm~\ref{alg:basic}.

%\newpage
\begin{algo}[Algorithm to solve~\eqref{eq:mainproblem}]\label{alg:basic} Let $z \in \cT$,  $w \in \cT \backslash \{0\}$, and $\varepsilon \in \vR_{++}$.

{\bf Initialize:}
\begin{enumerate}
\item  $w \leftarrow w_{\cG(z)}$;
\end{enumerate}

{\bf Repeat:} 
\begin{enumerate}
\item \label{alg:basic:step:computation} Computation: 
\begin{enumerate}
\item \label{alg:basic:step:r} $r \leftarrow \min\left\{\frac{z_i - z_{i+1}}{w_i - w_{i+1}} \mid i = 1, \ldots, n-1\right\}$ (where $0/0 = \infty$);
\item Define 
\begin{align*}
\lambda_0 \leftarrow \frac{\sum_{\{i \mid z_i > z_n\}} z_i w_i - \varepsilon}{\sum_{\{i \mid z_i > z_n\}} w_i^2} && \text{and} && \lambda_1 \leftarrow \frac{\dotp{z, w} - \varepsilon}{\|w\|^2};
\end{align*}
\item $n' \leftarrow \min\left(\{ k \mid z_k - \lambda_0 w_k < 0\}\cup \{n+1\}\right)$;
\item $\cG_0(z) \leftarrow \{G \in \cG(z) \mid \max(G) < n'\} \cup \{ \{n', \ldots, n\}\}$
\end{enumerate}
\item Tests:
\begin{enumerate}
\item \label{alg:basic:satisfied} If $\dotp{z,w} \leq \varepsilon$, set
\begin{enumerate}
\item $x^\ast \leftarrow z$.
\end{enumerate}
Exit;
\item \label{alg:basic:simplex} If $r = \infty$, set
\begin{enumerate}
\item   $x^\ast \leftarrow P_{\Delta(\varepsilon/w_1, n)}(z)$.
\end{enumerate}
Exit;
\item \label{alg:basic:lambda_1} If $\lambda_1 > r$, set 
\begin{enumerate}
\item $z \leftarrow z_{\cG(z - \lambda_1 w_0)}$;
\item $w \leftarrow w_{\cG(z - \lambda_1 w_0)}$; 
\end{enumerate}
Go to step~\ref{alg:basic:step:computation}.
\item \label{alg:basic:lambda_1z_n}If $\lambda_1 \leq r< \infty$ and $z_n - \lambda_1 w_{n} \geq 0$, set 
\begin{enumerate}
\item $x^\ast \leftarrow z - \lambda_1 w$ 
\end{enumerate} 
Exit;
\item \label{alg:basic:lambda_1z_nlambda_0g} If $\lambda_1 \leq r < \infty $, $z_n - \lambda_1 w_{n} < 0$ and $\lambda_0 > r$, set
\begin{enumerate}
\item $z \leftarrow z_{\cG(z - \lambda_0 w_0)}$;
\item $w \leftarrow w_{\cG(z - \lambda_0 w_0)}$; 
\end{enumerate}
Go to step~\ref{alg:basic:step:computation}.
\item \label{alg:basic:lambda_1z_nlambda_0l} If $\lambda_1 \leq r< \infty$, $z_n - \lambda_1 w_{n} < 0$, $\lambda_0 \leq r$, and $n' \leq n$ with $z_{n'} = z_n$, set 
\begin{enumerate}
\item $x^\ast \leftarrow \max\{z - \lambda_0w, 0\}$.
\end{enumerate}
Exit;
\item  \label{alg:basic:G_0} If $\lambda_1 \leq r < \infty$, $z_n - \lambda_1 w_{n} < 0$, $\lambda_0 \leq r$, and   $n' < n$ with  $z_{n'} \neq z_n$, set
\begin{enumerate}
\item $z \leftarrow z_{\cG_0}$;
\item $w \leftarrow w_{\cG_0}$; 
\end{enumerate}
Go to step~\ref{alg:basic:step:computation}.
%\item \label{prop:incrgamma:allfail} It cannot be the case that $\lambda_1 \leq r < \infty $, $z_n - \lambda_1 w_{n} < 0$, $\lambda_0 \leq r$, and  $n' = n+1$.
\end{enumerate}
\end{enumerate}

{\bf Output:} $x^\ast$.
\end{algo}

With the previous results, the following theorem is almost immediate.
\begin{theorem}
Algorithm~\ref{alg:basic} converges to $x^\ast$ in at most $n$ outer loops.
\end{theorem}
\begin{proof}
By Proposition~\ref{prop:equivalence}, $x^\ast = P_{H(w, \varepsilon) \cap \cT}(z)  = P_{H(w_{\cG(z)}, \varepsilon) \cap \cT}(z_{\cG(z)}) =   P_{H(w_{\cG(z)}, \varepsilon) \cap \cT}(z) $, so we can assume that $w = w_{\cG(z)}$ from the start. Furthermore, throughout this process $z$ and $w$ are updated to maintain that $\cG(z)\preccurlyeq G^\ast$, and so we can apply Proposition~\ref{prop:incrgamma} at every iteration. In particular, Proposition~\ref{prop:incrgamma} implies that during every iteration of Algorithm~\ref{alg:basic}, $z$ and $w$ must pass exactly one test. If tests~\ref{alg:basic:satisfied}, \ref{alg:basic:simplex},~\ref{alg:basic:lambda_1z_n}, or~\ref{alg:basic:lambda_1z_nlambda_0l} are passed, the algorithm terminates with the correct solution. If tests~\ref{alg:basic:lambda_1},~\ref{alg:basic:lambda_1z_nlambda_0g}, or~\ref{alg:basic:G_0} are passed, then we update $z$ and $w$, and the set $\cG(z)$ decreases in size by at least one. Because $1 \leq |\cG(z)| \leq n$, this process must terminate in at most $n$ outer loops.
%
%
%
%Then either $w$ is a constant vector and, hence, $r = \infty$, or $z$ and $w$  satisfy the conditions of Proposition~\ref{prop:incrgamma}.  In the former case, we can terminate the algorithm with the projection onto the simplex. Otherwise, we try steps~\ref{alg:basic:eq:test1}  and~\ref{alg:basic:eq:test2}  of the algorithm. If both of these steps fail then we know that $\lambda_{(z, w)}^\ast \geq r$. Thus, using the identity $x^\ast = P_{H_{w, \varepsilon} \cap \cT}(z) = P_{H_{w_{\cG(z- r w)}, \varepsilon} \cap \cT}(z_{\cG(z- r w)})$, we run the outer loop of the algorithm again, with knowledge that the solution has not changed. By the definition of $r$, we have $1 \leq |\cG(z- r w)| \leq |\cG(z)|-1$. The set $\cG(z)$ has at most $n$ elements at the start of the algorithm, so in total we can run at most $n$ outer loops until $|\cG(z)| = |\cG^\ast|$, and thus, $\cG(z) = \cG^\ast$.
\iftechreport
  \qed
\else 
\fi 
\end{proof}

The naive implementation of Algorithm~\ref{alg:basic}  has worst case complexity bounded above by $O(n^2\log(n))$ because we must continually sort the ratios in Step~\ref{alg:basic:step:r} and update the vectors $z$ and $w$ through averaging in Algorithm~\ref{alg:basic}. However, it is possible to keep careful track of $\lambda_0, \lambda_1, r, z$, and $w$ and get an $O(n\log(n))$ implementation of Algorithm~\ref{alg:basic}.  In order to prove this, we need to use a data-structure that is similar to a relational database. 

\begin{theorem}[Complexity of Algorithm~\ref{alg:basic}]\label{thm:convergence}
There is an $O(n\log(n))$ implementation of Algorithm~\ref{alg:basic}.
\end{theorem}
\begin{proof}
The key idea is to introduce a data structure $T_\cG = \{t_{G_1}, \ldots, t_{G_g}\}$ consisting of $5$-tuples, one for each group in a given partition $\cG = \{G_1, \ldots, G_g\}$:
\begin{align*}
\forall i \in \{1, \ldots, g\} \quad\quad t_{G_i} := (r_{G_i}, \min(G_i), \max(G_i), S(G_i, z), S(G_i, w)) \in \vR \times \vN \times \vN \times \vR^2
\end{align*}
where for any vector $x \in \vR^n$, we let $S(G, x) = \sum_{i\in G} x_i$, and the ratios $r_G$ are defined by
\begin{align*}
\forall i \in \{1, \ldots, g-1\} \quad\quad r_{G_i} := \frac{ \frac{S(G_i, z)}{|G_i|} 
- \frac{S(G_{i+1}, z)}{|G_{i+1}|} }{ 
\frac{S(G_i, w)}{|G_i|} 
- \frac{S(G_{i+1}, w)}{|G_{i+1}|} }, && \text{and} && r_{G_g} = \infty.
\end{align*}
Notice that $S(G, z) = S(G, z_\cG)$ and $S(G, w) = S(G, w_\cG)$. We assume that the data structure $T_\cG$ maintains 2 ordered-set views of the underlying tuples $t_G$, one of which is ordered by $r_G$, and another that is ordered by $\min(G)$. We also assume that the data structure allows us to convert iterators between views in constant time. This ensures that we can find the position of $t_G$ with $G \in \argmin\{r_G \mid G \in \cG\}$ in the view ordered by $r_G$ in time $O(\log(|\cG|)$ and convert this to an iterator (at the tuple $t_G$) in the view ordered by $\min(G)$ in constant time. We also assume that the ``delete," ``find," and ``insert," operations have complexity $O(\log(|\cG|))$. We note that this functionality can be implemented with the Boost Multi-Index Containers Library~\cite{boost}.

Now, the first step of Algorithm~\ref{alg:basic} is to build the data structure $T_{\cG(z)}$, which requires $O(n\log(n))$ operations. The remaining steps of the algorithm simply modify $T_{\cG(z)}$ by merging and deleting tuples. Suppose that Algorithm~\ref{alg:basic} terminates in $K$ steps for some $K \in \{1, \ldots, n\}$. For $i = 1, \ldots, K$, let $\cG_i$ be partition at the current iteration, and let $m_i = |\cG_i|$. Notice that for $i < K$, we have $\cG_{i} \preccurlyeq \cG_{i+1}$, so we get $\cG_{i+1}$ by merging groups in $\cG_i$, and $m_i > m_{i+1}$.  Finally, we also maintain two numbers throughout the algorithm: $I_{\cG_i} = \dotp{z_{\cG_i}, w_{\cG_i}}$ and $N_{\cG_i} = \|w_{\cG_i}\|^2$. Given $I_{\cG_i}$ and $N_{\cG_i}$, we can compute $\lambda_1$ and $\lambda_0$ in constant time.

Now fix $i \in \{1, \ldots, K-1\}$. Suppose that we get from iteration $i$ to $i+1$ through one of the updates $\cG_{i+1} = \cG(z_{\cG_i} - \lambda_1 w_{\cG_i})$ or $\cG_{i+1} = \cG(z_{\cG_i} - \lambda_0 w_{\cG_i})$. We note that each of these updates to $T_{\cG_i}$ can be performed in at most $O((m_i - m_{i+1})\log(m_i))$ steps because we call at most $O(m_i - m_{i+1})$ ``find", ``insert", ``delete", and ``merge" operations on the structure $T_{\cG_i}$ to get $T_{\cG_{i+1}}$, and at most $O(m_i - m_{i+1})$ modifications to the variables $I_{\cG_i}$ and $N_{\cG_i}$ to get $I_{\cG_{i+1}}$ and $N_{\cG_{i+1}}$. Likewise, it is easy to see that modifications of the form $\cG_{i+1} \leftarrow \cG_0(z_{\cG_i})$ can be implemented to run in $O((m_i - m_{i+1})\log(m_i))$ time.

Therefore, the total complexity of Algorithm~\ref{alg:basic} is 
\begin{align*}
O\left(n \log(n) + \sum_{i=1}^K (m_i - m_{i+1}) \log(m_i)\right) &= O\left(n \log(n) + \sum_{i=1}^K (m_i - m_{i+1}) \log(n)\right) \\
&= O(n \log(n)).
\end{align*}
%While the minimum ratio $r_G$ in $T_{\cG_i}$ is smaller than $\lambda$, we get an iterator to this minimal element $t_{G_{j_0}}$, switch to the view ordered by $\min(G)$, delete $t_{G_{j_0}}$ from the data structure and get a iterator to the next element of $t_{G_{j_0+1}}$. 
\iftechreport
  \qed
\else 
\fi 
\end{proof}

\section{Numerical Results}\label{sec:numerical}
In this section we present some numerical experiments to demonstrate the utility of the OWL norm and test our C++ implementation and MATLAB MEX file wrapper. 

\subsection{Synthetic Regression Test}
We adopt and slightly modify the experimental set up of~\cite[Section V.A]{zeng2014ordered}. We choose an integer $d \geq 1$, and generate a vector 
\begin{align*}
x_{\mathrm{true}} :=(\underbrace{0, \ldots 0}_{150d},\underbrace{3, \ldots, 3}_{50d}, \underbrace{0, \ldots 0}_{250d}, \underbrace{-4, \ldots, -4}_{50d},  \underbrace{0, \ldots 0}_{250d}, \underbrace{6, \ldots, 6}_{50d}, \underbrace{0, \ldots 0}_{250d})^T\in \vR^{1000d}.
\end{align*}
We generate a random matrix $A = [A_1, \ldots, A_{1000d}] \in \vR^{1000d \times 1000d}$ where the columns $A_i \in \vR^{1000d}$ follow a multivariate Gaussian distribution with $\text{cov}(A_i, A_j) = .8^{|i -j|}$ after which the columns are standardized and centered. Then we generate a measurement vector $b = Ax_{\mathrm{true}} + \nu$ where $\nu$ is Gaussian noise with variance $.01$. Next we generate $w$ with OSCAR parameters $\mu_1 = 10^{-3}$ and $\mu_2^{-5}$ (See Equation~\eqref{eq:OSCAR}).  Finally, we set $\varepsilon = \Omega_w(x_{\mathrm{true}})$.  

To test our implementation, we solve the regression problem
\begin{align}\label{eq:OWLREGI}
\Min_{x\in \vR^n} \frac{1}{2}\|Ax - b\|^2 \quad\quad \text{subject to: }  \Omega_w(x)\leq \varepsilon. 
\end{align}
with three different proximal splitting algorithms. We plot the results in Figure~\ref{fig:synthetictest}. 

\begin{figure}
\centering
\subcaptionbox{%
    \label{fig:synthetictest:subfig:itvsobj}%
}
[%
    0.45\textwidth % width of caption
]%
{%
    \includegraphics[width=0.45\textwidth]%
    {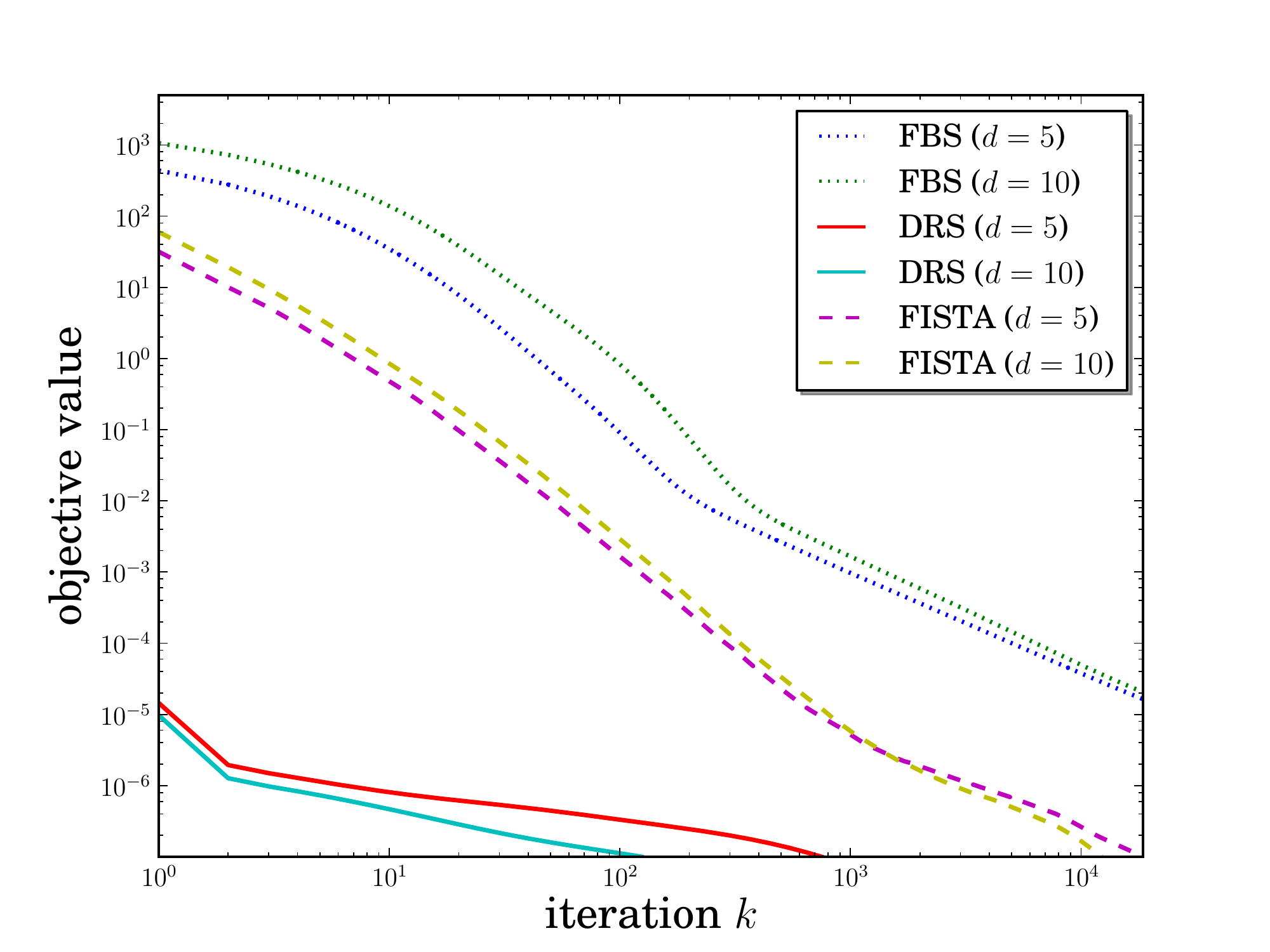}%
}%
\hspace{0.03\textwidth} % seperation
\subcaptionbox{%
    %}
    \label{fig:synthetictest:subfig:sublabel:timevsobj}%
}
[%
    0.45\textwidth % width of caption
]%
{%
   \includegraphics[width=0.45\textwidth]%
   {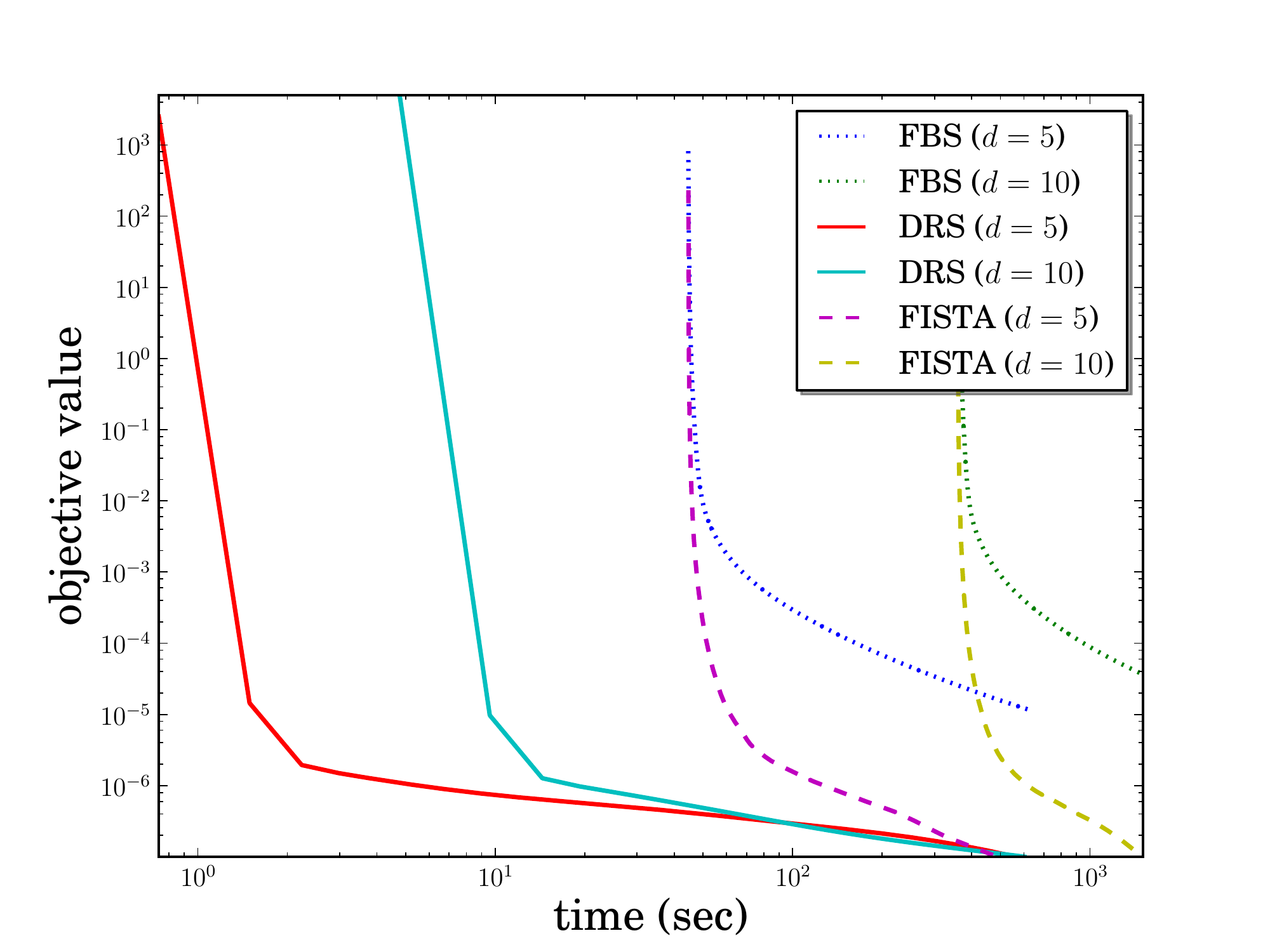}%
}%

\caption[Short Caption]{We solve Problem~\eqref{eq:OWLREGI} for $d=5,10$ with Douglas-Rachford splitting (DRS)~\cite{lions1979splitting}, Forward-Backward splitting (FBS)~\cite{passty1979ergodic}, and an accelerated forward-backward splitting method (dubbed FISTA~\cite{fista}). Note that the optimal objective value is $0$ because $\varepsilon = \Omega_{w}(x_{\mathrm{true}})$. In Figure~\ref{fig:synthetictest:subfig:sublabel:timevsobj}, there is a delay in the FBS and FISTA methods due to an initial investment in computing $\|A\|$, which is quite expensive. The test was run on a PC with 32GB memory and an Intel i5-3570 CPU with Ubuntu 12.04 and Matlab R2011b.}
\label{fig:synthetictest}
\end{figure}

\subsection{Standalone Tests}

In Table~\ref{tab:gaussian} we display the timings for our MATLAB MEX implementation of Algorithm~\ref{alg:basic}. Note that solutions to~\eqref{eq:OWLREGI} can be quite sparse (although usually not as sparse as solutions to~\eqref{eq:l1}). Thus, the iterates generated by algorithms that solve~\eqref{eq:OWLREGI}, such as those applied in Figure~\ref{fig:synthetictest}, are sparse as well. Thus, we test our implementation on high-dimensional vectors of varying sparsity levels.
\begin{center}
\begin{table}
\centering
    \begin{tabular}{lllll}
    \toprule
     Density & \multicolumn{4}{c}{length $n$} \\
\toprule
 & $10^3$ & $10^4$ & $10^5$ & $10^6$  \\ \toprule
    $100\%$ & $3.6$e-$04$ & $5.1$e-$03$ & $6.8$e-$02$ & $1.6$  \\\midrule
    $50\%$ & $2.1$e-$04$ & $3.1$e-$03$ & $3.8$e-$02$ & $8.3$e-$01$ \\  \midrule
    $25\%$ & $1.1$e-$04$ & $1.6$e-$03$ & $2.0$e-$02$ &$3.7$e-$01$ \\ \midrule
    $10\%$ & $5.6$e-$05$ & $8.5$e-$04$ & $1.0$e-$02$ & $1.4$e-$01$ \\
    \bottomrule
    \end{tabular}
 \caption{Average timings in seconds (over 100 runs) for random Gaussian vectors with different density levels (measured in percentage of nonzero entries). The test was run on a PC with 32GB memory and an Intel i5-3570 CPU with Ubuntu 12.04 and Matlab R2011b.}\label{tab:gaussian}
%    \begin{tabular}{lllll}
%    \toprule
%     $C$ & \multicolumn{4}{c}{kernel parameter $\sigma$} \\
%\toprule
% & $2^{-5}$ & $2^{-3}$ & $2^{-1}$ & $2$  \\ \toprule
%    $1$ & $0.82689$ & \cellcolor{blue!25} $0.83636$ & $0.82782$ & $0.7755$  \\\midrule
%    $2^2$ & $0.82658$ & $0.82441$ & $0.81742$ & $0.7755$ \\  \midrule
%    $2^4$ & $0.83465$ & $0.81835$ & $0.8168$ &$0.7755$ \\ \midrule
%    $2^6$ & $0.83465$ & $0.81835$ & $0.80795$ & $0.7755$ \\
%    \bottomrule
%    \end{tabular}
% \caption{Average timings (over 100 runs) for sparse random gaussian vectors}\label{tab:sparse}
\end{table}

\end{center}

\section{Conclusion}
In this paper, we introduced an $O(n\log(n))$ algorithm to project onto the OWL norm ball. Previously, there was no algorithm to compute this projection in a finite number of steps. We also evaluated our algorithm with a synthetic regression test. A C++ implementation of our algorithm with a MEX wrapper, \iftechreport
is 
\else 
will be made 
\fi 
available at the authors' website. 

\section*{Acknowledgement}
We thank Professor Wotao Yin for reading this draft and suggesting several improvements, Professor Robert Nowak for introducing us to the OWL norm,  Zhimin Peng for discussing code implementation issues with us, and Professor M{\'a}rio Figueiredo for trying out our code. 
% SIAM \bibliographystyle{siam}

\bibliographystyle{spmpsci}
\bibliography{bibliography}

\section*{Appendix}
\appendix
\section{Proof of Proposition~\ref{prop:equivalence}}\label{app:prop:equivalence}
First we prove a simple fact that we will use throughout the following proofs. Intuitively, it states that $\cG_1 \preccurlyeq \cG_2$ if, and only if, $\cG_2$ does not split groups in $\cG_1$.
\begin{lemma}[Equivalent conditions for nested partitions]\label{lem:equivSucc}
Let $\cG_1, \cG_2 \in \cP_n$. Then $\cG_1 \preccurlyeq \cG_2$ if, and only if, for every $i  \in \{1, \ldots, n\}$ such that there exists a group $G_1 \in \cG_1 $ with $i, i+1 \in G_1$, there exists a group $G_2$ such that $i, i+1 \in G_2$.
\end{lemma}
\begin{proof}[Proof of lemma]
$\implies$: This direction is clear by definition of $\preccurlyeq$.

$\impliedby$: Suppose that $G_1 = \{i_1, \ldots, i_k\} \in \cG_1$ for some $k \geq 1$. If $|G_1| = 1$, the partition property implies there exists $G_2 \in \cG_2$ containing $G_1$. Suppose $|G_1| > 1$. For each $i_j$ with $j = 1, \ldots, k-1$, there exists $G_2^j \in \cG_2$ with $i_j, i_{j+1} \in G_2^j$. Notice that each of the adjacent $G_2^j$ sets intersect: $i_{j}\in G_2^{j-1} \cap G_2^{j}$ for $j = 2, \ldots, k-1$. Thus, by the partition property, all $G_2^j$ are the same and hence, $G_1 \subseteq G_2^j$ for any such $j$. Thus, $\cG_1 \preccurlyeq \cG_2$.
\iftechreport
  \qed
\else 
\fi 
\end{proof}

Part~\ref{prop:equivalence:part:1}: Let $i \in \{1, \ldots, n\}$. Suppose that $z_i  = z_{i+1}$. Then $z_i - z_{i+1} = 0 \leq \lambda(w_i - w_{i+1})$, i.e., $z_i - \lambda w_{i} \leq z_{i+1} - \lambda w_{i+1}$. Therefore, by Lemma~\ref{lem:equivSucc}, we have $\cG(z) \preccurlyeq \cG(z - \lambda w)$. 

Next, suppose that $z_i - \lambda w_{i} \leq z_{i+1} - \lambda w_{i+1}$ where $z_i$ is not necessarily equal to $z_{i+1}.$ Then $i$ and $i+1$ are in the same group in $\cG(z - \lambda w)$. Thus, by Equation~\eqref{eq:averagedidentity}, we have $(z_{\cG(z - \lambda w)})_i = (z_{\cG(z - \lambda w)})_{i+1}$. Therefore, by Lemma~\ref{lem:equivSucc}, we have  $\cG(z - \lambda w) \preccurlyeq \cG(z_{\cG(z - \lambda w)})$. Conversely, suppose that $(z_{\cG(z - \lambda w)})_i = (z_{\cG(z - \lambda w)})_{i+1}$, but $z_i - \lambda w_i > z_{i+1} - \lambda w_{i+1}$. Then $i$ and $i+1$ are not in the same group in $\cG(z - \lambda w)$ and, in particular, $z_i - z_{i+1} > \lambda(w_i - w_{i+1}) \geq 0$. Thus, because $z\in \cT$, we have $(z_{\cG(z - \lambda w)})_i \geq z_i > z_{i+1} \geq (z_{\cG(z - \lambda w)})_{i+1}$, which is a contradiction. Therefore, by Lemma~\ref{lem:equivSucc}, we have $\cG(z - \lambda w) \succcurlyeq \cG(z_{\cG(z - \lambda w)})$, and so $\cG(z - \lambda w) = \cG(z_{\cG(z - \lambda w)})$.

Finally, suppose that there exists $i \in \{1, \ldots, n-1\}$ such that $z_i - \lambda w_i \leq z_{i+1} - \lambda w_{i+1}$. Then by Proposition~\ref{prop:optimality},
\begin{align*}
x_i^\ast - x_{i+1}^\ast := (z_i - \lambda^\ast w_i) - (z_{i+1} - \lambda^\ast w_{i+1}) + 2v^\ast_i - (v^\ast_{i-1} + v^\ast_{i+1}).
\end{align*}
If $x_i^\ast \neq x_{i+1}^\ast$, then $ 2v^\ast_i = 0$ so the expression on the left is nonpositive, which is a contradiction. Thus, $x_i^\ast = x_{i+1}^\ast$. Therefore, by Lemma~\ref{lem:equivSucc}, we have $\cG(z_{\cG(z - \lambda w)}) = \cG(z - \lambda w) \preccurlyeq \cG^\ast$. 

Part~\ref{prop:equivalence:part:2}:  %To prove the second equality, suppose that $G \in \cG(z- \lambda w)$. 
%Then for all $i, j \in G$, we have $x_i^\ast = x_{j}^\ast$, so by averaging the optimality conditions in Propositions~\ref{prop:optimality}, we have
%\begin{align*}
%x_i^\ast &= \frac{1}{|G|}\sum_{j \in G} \left(z_j - \lambda_{(z, w)}^\ast w_j + (v^\ast_{(z, w)})_{j} - (v^\ast_{(z, w)})_{j-1}\right) \\
%&= (z_{\cG(z - \lambda w)})_{i} - \lambda_{(z, w)}^\ast (w_{\cG(z - \lambda w)})_{i} + \frac{1}{|G|}\left((v^\ast_{(z, w)})_{\max{G}} - (v^\ast_{(z, w)})_{\min{G}}\right).
%\end{align*}
%Note that  $(1/|G|)(v^\ast_{(z, w)})_{\max{G}}(x_{i}^\ast - x_{i+1}^\ast) = 0$ (when $i < n$) and $(1/|G|)(v^\ast_{(z, w)})_{\min{G}}(x_{i}^\ast - x_{i-1}^\ast) = 0$ for all $i \in G$. In addition, if $n \in G$, we have $(v^\ast_{(z, w)})_{\max(G)}x_n^\ast= (v^\ast_{(z, w)})_n x_n^\ast = 0$. 
Note that
\begin{align*}
\dotp{w_{\cG}, x^\ast} = \sum_{G \in \cG} \sum_{i \in G} x_i^\ast\left( \frac{1}{|G|} \sum_{j \in G} w_j\right) = \dotp{w, x^\ast} = \varepsilon
\end{align*}
because $x^\ast$ is constant along each group $G$. Thus, $x^\ast \in H(w_{\cG}, \varepsilon)\cap \cT$. Let $x^0 = P_{H(w_{\cG}, \varepsilon) \cap \cT}(z_{\cG})$. We will show that $x^\ast = x^0$. Indeed, $\cG  = \cG(z_{\cG}) \preccurlyeq \cG(x^0)$ and 
\begin{align*}
\dotp{w, x^0} =  \sum_{G \in \cG} \sum_{i \in G} x^0_i\left( \frac{1}{|G|} \sum_{j \in G} w_j\right)=  \dotp{w_{\cG}, x^0} = \varepsilon
\end{align*}
because $x^0$ is constant along each group. Therefore, $x^0 \in H(w, \varepsilon) \cap \cT$.  In addition, for all $G \in \cG$, we have $x^0_i = x^0_j$ for all $i,j \in G$; let $x^G$ denote $x^0_i$ for any $i \in G$. Therefore, 
\begin{align*}
\|z - x^\ast\|^2 \leq \|z - x^0\|^2 = \sum_{G \in \cG} \sum_{i \in G}\left(z_i - x^G\right)^2 &= \sum_{G \in \cG}\left(\frac{1}{2|G|} \sum_{i, j \in G} (z_i - z_j)^2 + \sum_{i \in G} \left(z_{\cG} - x^G\right)^2\right) \\
&\leq \sum_{G \in \cG}\left(\frac{1}{2|G|} \sum_{i, j \in G} (z_i - z_j)^2 + \sum_{i \in G} \left(z_{\cG} - x_i^\ast\right)^2\right) \\
&= \|z - x^\ast\|^2
\end{align*}
Thus, $\|z - x^\ast\| = \|z - x^0\|$, so by the uniqueness of the projection, we have $x^0 = x^\ast$.

%Because $\lambda^\ast_{(z, w)} > 0$ so there exists $0 < \lambda < \lambda^\ast_{(z, w)}$ small enough that $\cG = \cG(z)$. Thus, the third equality follows by second equality after noting that $z_{\cG(z)} = z$.

\section{Proof of Proposition~\ref{prop:incrgamma}}\label{app:prop:incrgamma}

First note that because $\dotp{w, x^\ast} = \varepsilon$,  Proposition~\ref{prop:optimality} implies that
\begin{align*}
 \lambda^\ast &= \frac{\sum_{i=1}^n z_i w_i - \varepsilon + \sum_{i=1}^{n-1} v^\ast_i(w_i - w_{i+1}) + v_n^\ast w_n^\ast}{\|w\|^2} \geq \lambda_1. \numberthis\label{eq:lambda1ident}
\end{align*}
because $\sum_{i=1}^{n-1} v^\ast_i(w_i - w_{i+1}) + v_n^\ast w_n^\ast \geq 0$.

Part~\ref{prop:incrgamma:simplex}: Suppose $r = \infty$. Then $w$ is a constant vector. Thus, the result follows from Proposition~\ref{prop:simplex}.

Part~\ref{prop:incrgamma:lambda_1}: Suppose that $\lambda_1 > r$. Then $\lambda^\ast > r$ by Equation~\eqref{eq:lambda1ident}.

Part~\ref{prop:incrgamma:lambda_1z_n}: Suppose that $\infty > r \geq \lambda_1$ and $z_n - \lambda_1 w_n \geq 0$. Then $z - \lambda_1 w \in \cT$ and $x^0 = z - \lambda_1 w$ satisfies the conditions of Proposition~\ref{prop:optimality} with $v^\ast = 0$ and $\lambda^\ast = \lambda_1$. Thus, $x^\ast = z - \lambda_1 w$.

Part~\ref{prop:incrgamma:lambda_1z_nlambda_0g}: Suppose that $\infty > r \geq \lambda_1,$ $z_n - \lambda_1 w_n < 0$, and $\lambda_0 > r$. Then, $z_n - \lambda^\ast w_n \leq z_n - \lambda_1 w_n < 0$.  From $x_n^\ast = z_i - \lambda^\ast w_n + v^\ast_n - v^\ast_{i-1} < v^\ast_n$, and $v^\ast_nx_n^\ast = 0$, we have $x_n^\ast = 0$. Next, because $\cG(z) \preccurlyeq \cG^\ast$, we have $\{i \mid z_i = z_n\} \subseteq \{i \mid x_i^\ast = x_n^\ast\} = \{i \mid x_i^\ast =0\}$ and so $\{i \mid z_i > z_n \} \supseteq \{i \mid x_i^\ast > 0\}$. Let $k_0 = \max\{i \mid z_i > z_n \}$. Therefore, from $\sum_{\{i \mid z_i > z_n\}} x_i^\ast w_i= \sum_{\{i \mid x_i^\ast > 0\}} x_i^\ast w_i  = \varepsilon$ and Proposition~\ref{prop:optimality}, we have
\begin{align*}
 \lambda^\ast &=  \frac{\sum_{\{i \mid z_i > z_n\}} z_i w_i - \varepsilon + \sum_{\{i \mid z_i > z_n\}} v^\ast_i(w_i - w_{i+1}) + v_{k_0}w_{k_0 + 1}}{\sum_{\{i \mid z_i > z_n\}} w_i^2} \geq \lambda_0 > r \numberthis\label{eq:lambda0ident}
\end{align*}
where we use the bound $\sum_{\{i \mid z_i > z_n\}} v^\ast_i(w_i - w_{i+1}) + v_{k_0}w_{k_0 + 1} \geq 0$. Notice that $x_n^\ast = 0$ and the first inequality in Equation~\eqref{eq:lambda0ident} holds whether or not $\lambda_0 > r$: we just need $\infty > r \geq \lambda_1$ and $z_n - \lambda_1 w_n < 0$. We will use this fact in Part~\ref{prop:incrgamma:G_0} below.

Part~\ref{prop:incrgamma:lambda_1z_nlambda_0l}: Suppose that $\infty > r \geq \lambda_1,$ $z_n - \lambda_1 w_n < 0$, $r \geq \lambda_0$, $n' \leq n$ and $z_{n'} = z_n$. Then $\max\{z - \lambda_0 w, 0\} \in \cT$. In addition, we have $\dotp{w, \max\{z- \lambda_0 w, 0\}} = \varepsilon$ by the choice of $\lambda_0$. We will now define a vector $v \in \vR^n_{+}$ recursively: If $z_i > z_n$, set $v_i = 0$; otherwise set $v_i = v_{i-1} - (z_i - \lambda_0 w_i)$. We can satisfy the optimality conditions of Proposition~\ref{prop:optimality} with $\lambda^\ast = \lambda_0$ and $v^\ast = v$. Thus, $x^\ast = \max\{z - \lambda_0 w, 0\}$.

Part~\ref{prop:incrgamma:G_0}: Suppose that $\infty > r \geq \lambda_1,$ $z_n - \lambda_1 w_n < 0$, $r \geq \lambda_0$, $n' < n$ and $z_{n'} \neq z_n$.  From the proof of Part~\ref{prop:incrgamma:lambda_1z_nlambda_0g} we have $z_{k} - \lambda^{\ast} w_{k} \leq z_{k} - \lambda_0w_{k}< 0$ for all $k = n', n' + 1, \ldots, n$ (from $\lambda^\ast \geq \lambda_0$) and $x_{n}^\ast = 0$. Suppose that $x_{n'}^\ast \neq x_n^\ast = 0$. Let $n'' = \min\{ k \mid x_k^\ast = 0\}$. Then $n'' -1 \geq n' \geq 1$.  Thus because $x_{n'' -1}^\ast \neq x_{n''}^\ast = 0$, we have $v^\ast_{n'' - 1}= 0$ and $x_{n'' - 1}^\ast = z_{n'' - 1} - \lambda^\ast w_{n'' - 1} - v^\ast_{n'' - 2} < 0$ (where we let $v_{n'' - 2}^\ast = 0$ if $n'' = 2$). This is a contradiction because $x^\ast \in \cT$. Thus, $x_{n'}^\ast = x_{n' + 1}^\ast = \ldots = x_n^\ast = 0$.   If $n' = 1$, then we see that $\cG(z) \preccurlyeq \cG_0 \preccurlyeq \cG^\ast$. Furthermore, if $n' > 1$, then we claim that $n' -1$ and $n'$ are not in the same group in $\cG(z)$, i.e., that $z_{n'-1} \neq z_{n'}$. Indeed, if $z_{n'-1} = z_{n'}$, then $w_{n'-1} = w_{n'}$ and hence, $z_{n' -1} - \lambda_0 w_{n'-1} = z_{n'} - \lambda_0 w_{n'} < 0$, which is a contradiction. 

Thus, this argument has shown that $\cG(z) = \{G\in \cG(z) \mid \max(G) < n'\} \cup \{G \in \cG(z) \mid \min(G) \geq n'\}$ and there exists $G_2 \in \cG^\ast$ with $\{n', \ldots, n\} \subseteq G_2$. Note that the first of these identities implies that $\cG_0\in \cP_n$. Let us now prove the claimed nestings: $\cG(z) \preccurlyeq  \cG_0 = \cG(z_{\cG_0}) \preccurlyeq \cG^\ast$.
\begin{enumerate}
\item $(\cG(z) \preccurlyeq  \cG_0)$: 
Suppose that $G \in \cG(z)$. If $\max(G) < n'$, then $G \in \cG_0$. If $\min(G) \geq n'$, then $G \subseteq \{n', \ldots, n\} \in \cG_0$. Thus, $\cG(z) \preccurlyeq \cG_0$. 
\item $(\cG_0 = \cG(z_{\cG_0}))$: The identity follows because
\begin{align*}
(z_{\cG_0})_i = \begin{cases}
z_i &\text{if } i < n'; \\
\frac{1}{n - n' + 1} \sum_{i=n'}^n z_i & \text{if } i \geq n'.
\end{cases}
\end{align*}
%
%Suppose that $G \in \cG_0$. If $\max(G) < n$, then $G \in \cG(z)$. Thus, for all $i, j \in G$, we have $(z_{\cG_0})_i = z_i = z_j = (z_{\cG_0})_j$. Therefore, there is a $G_2 \in \cG(z_{\cG_0})$ such that $G \subseteq G_2$. If $\min(G) \geq n'$, then $G = \{n', \ldots, n\}$. Note that $(z_{\cG_0})_{n'} = \cdots = (z_{\cG_0})_{n}$. Thus, there exists $G_2 \in \cG(z_{\cG_0})$ with $G \subseteq G_2$. Therefore, we have shown that $\cG_0 \preccurlyeq \cG(z_{\cG_0})$.  Now suppose that $G \in \cG(z_{\cG_0})$. If $\max(G) < n'$, then $G \in \cG(z)$, so there exists $G_2 \in \cG_0$ with $G \subseteq G_2$. If $\max(G) > n'$, then as before, $G = \{n', \ldots, n\} \in \cG_0$.  Thus, we have shown that $\cG_0 \preccurlyeq \cG(z_{\cG_0})$ and $\cG_0 \succcurlyeq \cG(z_{\cG_0})$. Therefore, $\cG_0 = \cG(z_{\cG_0})$.
\item $(\cG_0 \preccurlyeq \cG^\ast)$: Suppose that $G \in \cG_0$. If $\max(G) < n'$, it follows that $G \in \cG(z)$ and hence by Part~\ref{prop:equivalence:part:1} of Proposition~\ref{prop:equivalence}, there is a $G_2 \in \cG^\ast$ with $G \subseteq G_2$. If $\min(G) \geq n'$, then $G = \{n', \ldots, n\}$ and there exists $G_2 \in \cG^\ast$ with $G \subseteq G_2$. Therefore, $G_0 \preccurlyeq \cG^\ast$. 
\end{enumerate}
%
%Now suppose that $G \in \cG(z)$. If $\max(G) < n$, then $G \in \cG_0$. If $\min(G) \geq n'$, then $G \subseteq \{n', \ldots, n\} \in \cG_0$, so $\cG(z) \preccurlyeq \cG_0$. No
%
%Because of the nesting $\cG(z) \preccurlyeq \cG^\ast$ (Part~\ref{prop:equivalence:part:1} of Proposition~\ref{prop:equivalence}), it now easily follows that $\cG_0\in \cP_n$ and $\cG(z) \preccurlyeq  \cG_0 = \cG(z_{\cG_0}) \preccurlyeq \cG^\ast$.  
Finally, note that $|\cG(z_{\cG_0})| = |\cG_0| \leq |\cG(z)| - 1$ because $z_{n'} \neq z_n$ implies that $\{G \in \cG(z) \mid \min(G) \geq n'\} \subseteq \cG(z)$ contains at least two distinct groups that are both contained in $\{n', \ldots,n\} \in \cG_0$. 

Part~\ref{prop:incrgamma:allfail}: Suppose that $\infty > r \geq \lambda_1,$ $z_n - \lambda_1 w_n < 0$, $ r > \lambda_0 $, and $n' = n+1$. Then $z_n - \lambda_0w_n \geq 0$. Thus, $\lambda_1 > \lambda_0$ and
\begin{align*}
\lambda_1 \left(\sum_{\{i \mid z_i > z_n\}} w_i^2 + \sum_{\{i \mid z_i = z_n\}} w_i^2\right) &= \left(\sum_{\{i \mid z_i > z_n\}} z_iw_i + \sum_{\{i \mid z_i = z_n\}} z_iw_i\right) - \varepsilon \\
&< \left(\sum_{\{i \mid z_i > z_n\}} z_iw_i -\varepsilon\right)  + \sum_{\{i \mid z_i = z_n\}} \lambda_1 w_i^2 \\
&= \lambda_0\left(\sum_{\{i \mid z_i > z_n\}} w_i^2 \right)+  \lambda_1 \sum_{\{i \mid z_i = z_n\}} w_i^2.
\end{align*}
where the strict inequality follows from $z_n < \lambda_1 w_n$. Thus,  $\lambda_1 < \lambda_0$, which is a contradiction.

The final conclusions of the proposition are simple consequence of Lemma~\ref{eq:equalcardinality}, Proposition~\ref{prop:equivalence}, and the 6 alternatives. 

\section{An Example}\label{app:example}

In this section, we project the point $z_0 = (3, 2, 1, -1, 2)$ onto the OWL ball of radius $\varepsilon = 1$ with weights $w_0 = (5, 4, 3, 1, 1)$.

\begin{itemize}
\item \textbf{Preprocessing.} 
\begin{itemize} 
\item Set $s := \sign(z_0) = (1,1, 1, -1, 1)^T$;
\item Set $z := Q(|z_0|)|z_0| = (3, 2, 2, 1, 1)^T$;
\item Set $\cG(z) \leftarrow \{\{1\}, \{2, 3\}, \{4, 5\}\}$;
\item Set $w := (w_0)_{\cG(z)} = (5, 7/2, 7/2, 1, 1)^T$;
\end{itemize}
\item \textbf{Iteration 1.}
\begin{itemize}
\item Set $$r \leftarrow \min\left\{\frac{3 - 2}{5 - 7/2}, \infty, \frac{2-1}{7/2 - 1},\infty \right\} = \frac{2}{5};$$
\item Set 
\begin{align*}
\lambda_0 \leftarrow \frac{28}{49.5} && \text{and} && \lambda_1 \leftarrow \frac{31}{51.5};
\end{align*}
\item Set $n' \leftarrow 6$;
\item Set $\cG_0(z) \leftarrow \cG(z)$;
\item \textbf{Test~\ref{alg:basic:lambda_1} passed:} $\lambda_1 = 31/51.5 > 2/5 = r$;
\begin{itemize}
\item Set $\cG(z - \lambda_1 w) \leftarrow \{\{1\}, \{2, 3, 4, 5\}\}$;
\item Set $z \leftarrow z_{\cG(z - \lambda_1 w)} = (3, 3/2, 3/2, 3/2, 3/2)^T$;
\item Set $w \leftarrow w_{\cG(z - \lambda_1 w)} = (5, 9/4, 9/4, 9/4, 9/4)^T$;
\end{itemize}
\end{itemize}
\item \textbf{Iteration 2.}
\begin{itemize}
\item Set $$r \leftarrow \min\left\{\frac{3 - 3/2}{5 - 9/4}, \infty, \infty, \infty \right\} = \frac{12}{22};$$
\item Set 
\begin{align*}
\lambda_0 \leftarrow \frac{14}{25} && \text{and} && \lambda_1 \leftarrow \frac{27.5}{45.25};
\end{align*}
\item Set $n' \leftarrow 6$;
\item Set $\cG_0(z) \leftarrow \cG(z) = \{\{1\}, \{2, 3, 4, 5\}\}$;
\item \textbf{Test~\ref{alg:basic:lambda_1} passed:} $\lambda_1 = 27.5/45.25 > 12/22 = r$;
\begin{itemize}
\item Set $\cG(z - \lambda_1 w) \leftarrow \{\{1, 2, 3, 4, 5\}\}$;
\item Set $z \leftarrow z_{\cG(z - \lambda_1 w)} = (9/5, 9/5, 9/5, 9/5, 9/5)^T$;
\item Set $w \leftarrow w_{\cG(z - \lambda_1 w)} = (14/5, 14/5, 14/5, 14/5, 14/5)^T$;
\end{itemize}
\end{itemize}
\item \textbf{Iteration 3.}
\begin{itemize}
\item Set $r = \infty$;
\item \textbf{Test~\ref{alg:basic:simplex} passed:} (We use Proposition~\ref{prop:simplex} to finish.)
\begin{itemize}
\item Set $\lambda = 121/70$;
\item Set $x^\ast = \max\{z - \lambda , 0\} = (1/14, 1/14, 1/14, 1/14, 1/14)^T$;
\end{itemize}
\end{itemize}
\item \textbf{Undo preprocessing.} 
\begin{itemize}
\item Set $x_0^\ast = s \odot Q(|z_0|)^T x^\ast = (1/14, 1/14, 1/14, -1/14, 1/14)^T$;
\end{itemize}
\item \textbf{Terminate.}
\begin{itemize}
\item We have $P_{\cB(w_0, \varepsilon)}(z_0) = x_0^\ast$.
\end{itemize}
\end{itemize}

Notice that $x_0^\ast$ satisfies $\Omega_{w_0}(x_0^\ast) = 1$ because $\sum_{i=1}^5 (w_0)_i = 14.$

\end{document}